\documentclass[preprint]{elsarticle}

\usepackage{amsfonts,amssymb,amsmath,amsthm}
\usepackage{graphicx}
\usepackage{epstopdf}
\usepackage{algorithmic}
\usepackage{amsopn}

\DeclareMathOperator{\tr}{tr} 

\newcommand{\R}{{\mathbb R}}

\newcommand{\cX}{{\sf X}}
\newcommand{\cY}{{\sf Y}}
\newcommand{\cZ}{{\sf 0}}
\newcommand{\cS}{{\sf S}}
\newcommand{\cP}{{\sf P}}

\newcommand{\al}{{\alpha}}
\newcommand{\be}{{\beta}}

\newtheorem{theorem}{Theorem}

\begin{document}

\title{The center problem for the Lotka reactions \\ with generalized mass-action kinetics}

\author[ricam]{Bal\'azs Boros}
\ead{balazs.boros@ricam.oeaw.ac.at}

\author[univie]{Josef Hofbauer}
\ead{josef.hofbauer@univie.ac.at}

\author[jku]{Georg Regensburger}
\ead{georg.regensburger@jku.at}

\author[ricam]{Stefan M\"uller\corref{cor}}
\ead{stefan.mueller@ricam.oeaw.ac.at}
\cortext[cor]{Corresponding author}

\address[ricam]{Radon Institute for Computational and Applied Mathematics, \\ Austrian Academy of Sciences, Linz, Austria}

\address[univie]{Department of Mathematics, University of Vienna, Austria}

\address[jku]{Institute for Algebra, Johannes Kepler University Linz, Austria}

\begin{abstract}
Chemical reaction networks with generalized mass-action kinetics lead to power-law dynamical systems.
As a simple example, we consider the Lotka reactions 
and the resulting planar ODE.
We characterize the parameters (positive coefficients and real exponents)
for which the unique positive equilibrium is a center.
\end{abstract}

\begin{keyword}
Chemical reaction network, power-law kinetics, center-focus problem, focal value, first integral, reversible system
\end{keyword}


\maketitle


\section{Introduction}

Lotka~\cite{lotka:1910} considered a series of three chemical reactions,
transforming a substrate into a product via two intermediates, $\cX$~and~$\cY$.
If the reactions producing $\cX$~and~$\cY$, respectively, are assumed to be autocatalytic,
then the resulting ODE 
is the classical Lotka-Volterra predator-prey system~\cite{lotka:1920:a,lotka:1920:b}.

Farkas and Noszticzius \cite{farkas:noszticzius:1985} and Dancs\'o et al.~\cite{dancso:farkas:farkas:szabo:1991}
considered generalized Lotka-Volterra schemes, 
arising from the Lotka reactions with power-law kinetics. 
They studied the ODE
\begin{align} \label{ode_intro}
\dot x &= k_1 \, x^{\hat p} - k_2 \, x^{p} y^{q} , \\
\dot y &= k_3 \, x^{p} y^{q} - k_4 \, y^{\hat q} \nonumber
\end{align}
with positive coefficients $k_1,k_2,k_3,k_4>0$ and real exponents $p,q,\hat p,\hat q\ge1$.
(The special case $p=q=\hat p=\hat q=1$ is the classical Lotka-Volterra system.)
Dancs\'o et al.~\cite{dancso:farkas:farkas:szabo:1991} provided a local stability and bifurcation analysis.
In particular, by finding first integrals,
they determined four cases where the ODE admits a center.

In this work,
we allow arbitrary real exponents $p,q,\hat p,\hat q\in\R$ in the ODE~\eqref{ode_intro}.
In addition to the four known cases,
we identify two new cases of centers,
by showing that they correspond to reversible systems.
Moreover,
we prove that centers are characterized by these six cases.

The paper is organized as follows.
In Section 2,
we elaborate on the chemical motivation of the ODE under study,
and in Section 3, we present our main result.

\section{The Lotka reactions with generalized mass-action kinetics}
As in the original work by Lotka \cite{lotka:1910},
we start by considering a series of net reactions, $\cS \to \cX$, $\cX \to \cY$, and $\cY \to \cP$,
which transform a substrate into a product.
We are interested in the dynamics of $\cX$~and~$\cY$ only,
in particular, we assume that the substrate is present in constant amount and that the product does not affect the dynamics.
As a consequence, we omit substrate and product from consideration and arrive at the simplified reactions
\[
\cZ \to \cX, \, \cX \to \cY, \, \cY \to \cZ .
\]
To obtain a classical Lotka-Volterra system as in~\cite{lotka:1920:a,lotka:1920:b},
one assumes the first and the second reaction to be autocatalytic,
in particular, one defines the kinetics of the reactions as
$v_{\cZ\to\cX} = k_{\cZ\to\cX} [\cX]$, $v_{\cX\to\cY} = k_{\cX\to\cY} [\cX][\cY]$, and $v_{\cY\to\cZ} = k_{\cY\to\cZ} [\cY]$
with rate constants $k_{\cZ\to\cX}, k_{\cX\to\cY}, k_{\cY\to\cZ} > 0$ and concentrations $[\cX],[\cY]\ge0$.
In this work, we consider the Lotka reactions with arbitrary power-law kinetics.
In terms of chemical reaction network theory,
we assume generalized mass-action kinetics~\cite{mueller:regensburger:2012,mueller:regensburger:2014},
that is,
\[
v_{\cZ \to \cX} = k_{\cZ \to \cX} [\cX]^{\al_1} [\cY]^{\be_1}, \quad
v_{\cX \to \cY} = k_{\cX \to \cY} [\cX]^{\al_2} [\cY]^{\be_2}, \quad
v_{\cY \to \cZ} = k_{\cY \to \cZ} [\cX]^{\al_3} [\cY]^{\be_3} ,
\]
with arbitrary real exponents $\al_1, \be_1, \al_2, \be_2, \al_3, \be_3 \in \R$.
The resulting ODE for the concentrations $x=[\cX]$ and $y=[\cY]$ amounts to
\begin{align} \label{ode_math}
\dot x &= k_1 \, x^{\al_1} y^{\be_1} - k_2 \, x^{\al_2} y^{\be_2} , \\
\dot y &= k_3 \, x^{\al_2} y^{\be_2} - k_4 \, x^{\al_3} y^{\be_3} , \nonumber
\end{align}
where $k_1=k_{\cZ\to \cX}$, $k_2=k_3=k_{\cX \to\cY}$, and $k_4=k_{\cY\to\cZ}$.
Since we allow real exponents,
we consider the dynamics on the positive quadrant.
In fact,
we study an ODE which is orbitally equivalent to \eqref{ode_math} on the positive quadrant and has two exponents less,
\begin{align} \label{ode_trafo}
\dot x &= k_1 \, x^{a_1} y^{b_1} - k_2 , \\
\dot y &= k_3 - k_4 \, x^{a_3} y^{b_3} , \nonumber
\end{align}
where $a_1=\al_1-\al_2$, $b_1=\be_1-\be_2$, $a_3=\al_3-\al_2$, $b_3=\be_3-\be_2$.
Further, we assume that the ODE admits a positive equilibrium $(x^*,y^*)$
and use the equilibrium to scale the ODE~\eqref{ode_trafo}.
We introduce $K = \frac{k_3}{k_2} \, \frac{x^*}{y^*}>0$
and obtain 
\begin{align} \label{ode_K}
\dot x &= x^{a_1} y^{b_1} - 1 , \\
\dot y &= K \left( 1 - x^{a_3} y^{b_3} \right) . \nonumber
\end{align}
Clearly, the ODE~\eqref{ode_K} admits the equilibrium $(1,1)$ which is not necessarily unique,
and the Jacobian matrix at $(1,1)$ is given by
\begin{equation} \label{eq:J}
J =
\begin{pmatrix}
a_1 & b_1 \\
-Ka_3 & -Kb_3
\end{pmatrix} .
\end{equation}
Dancs{\'o} et al.~\cite{dancso:farkas:farkas:szabo:1991} studied the ODE~\eqref{ode_K}
in the orbitally equivalent form
\begin{align} \label{ode_dancso}
\dot x &= x^{\hat p} - x^{p} y^{q} , \\
\dot y &= C \left( x^{p} y^{q} - y^{\hat q} \right) , \nonumber
\end{align}
where $\hat p=a_1-a_3$, $\hat q=b_3-b_1$, $p=-a_3$, $q=-b_1$, and $C=K$.
They stated four cases where the equilibrium $(1,1)$ is a center and provided first integrals.
In this work, we identify two new cases and show that they correspond to reversible systems.
Moreover, we prove that every center belongs to one of the six cases.


\section{Main result}
\label{sec:main}

An equilibrium is a {\em center} if all nearby orbits are closed.

\begin{theorem}
The following statements are equivalent.

\begin{itemize}
\item[\rm 1.]
The equilibrium $(1,1)$ of the ODE~\eqref{ode_K} with $K>0$ is a center.
\item[\rm 2.]
The eigenvalues of the Jacobian matrix at $(1,1)$ are purely imaginary, that is, $\tr J=0$ and $\det J>0$,
and the first two focal values vanish.
\item[\rm 3.]
The parameter values $a_1,b_1,a_3,b_3 \in \R$, and $K>0$ belong to one of the six cases in Table~\ref{table}.
\end{itemize}
\end{theorem}

\begin{proof}
1 $\Rightarrow$ 2: If $J$ has a zero eigenvalue, that is, $\det J=0$, then $(1,1)$ lies on a curve of equilibria
and cannot be a center.
Hence, 
the eigenvalues of $J$ are purely imaginary,
and all focal values vanish.

2 $\Rightarrow$ 3:
For the computation of the first two focal values, $L_1$ and $L_2$,
and the case distinction implied by $\tr J=0$, $\det J>0$, and $L_1=L_2=0$,
see Subsection~\ref{ss:cd}.

3 $\Rightarrow$ 1:
For the cases (i)--(iv) in Table~\ref{table},
first integrals have been given by Dancs\'o et al.~\cite{dancso:farkas:farkas:szabo:1991}.
In fact, they determined all the cases for which a first integral can be found by using an integrating factor of the form $x^Ay^B$.
See Table~\ref{table:fi} and \cite[p.~122, Table~I]{dancso:farkas:farkas:szabo:1991}.

Case (i) includes the classical Lotka-Volterra systems;
the corresponding first integral is of the type of separated variables and was already stated in Farkas and Noszticzius~\cite{farkas:noszticzius:1985}.
In case~(iv), there is a typo in \cite{dancso:farkas:farkas:szabo:1991}; the correct formula is $C=\tfrac{1}{\hat q - 1}$.

The remaining cases, (r1) and (r2), are reversible systems. See Subsection~\ref{ss:rs}.
\end{proof}

\renewcommand{\arraystretch}{2.5}

\begin{table} \label{table}
\begin{center}
\begin{tabular}{|c|c|c|c|c|}
\hline
case & \multicolumn{3}{c|}{parameters} & ODE \\[2ex]
\hline\hline
(i) & $a_1=b_3=0$
& $K>0$
& $a_3b_1>0$
& $\begin{aligned} \dot x &= y^{b_1}-1 \\ \dot y &= K(1-x^{a_3}) \end{aligned}$ \\[2ex] \hline
(ii) & $\begin{aligned} a_1&=a_3+1 \\ b_3&=b_1+1 \end{aligned}$
& $K=\textstyle \frac{a_1}{b_3}>0$
& $a_1+b_3<1$
& $\begin{aligned} \dot x &= x^{a_1}y^{b_3-1}-1 \\ \dot y &= \textstyle \frac{a_1}{b_3} (1-x^{a_1-1}y^{b_3}) \end{aligned}$ \\[2ex] \hline
(iii) & $\begin{aligned} a_3&=-1 \\ b_3&=1 \end{aligned}$
& $K=a_1>0$
& $a_1+b_1<0$
& $\begin{aligned} \dot x &= x^{a_1}y^{b_1}-1 \\ \dot y &= \textstyle a_1 (1-\textstyle \frac{y}{x}) \end{aligned}$ \\[2ex] \hline
(iv) & $\begin{aligned} a_1&=1 \\ b_1&=-1 \end{aligned}$
& $K=\textstyle \frac{1}{b_3}>0$
& $a_3+b_3<0$ 
& $\begin{aligned} \dot x &= \textstyle \frac{x}{y}-1 \\ \dot y &= \textstyle \frac{1}{b_3}(1-x^{a_3}y^{b_3}) \end{aligned}$ \\[2ex] \hline \hline
(r1) & $\begin{aligned} a_1&=b_3 \\ a_3&=b_1 \end{aligned}$
& $K=1$
& $|a_1|<|b_1|$
& $\begin{aligned} \dot x &= x^{a_1}y^{b_1}-1 \\ \dot y &= 1-x^{b_1}y^{a_1} \end{aligned}$ \\[2ex] \hline
(r2) & $\begin{aligned} a_1&=Kb_3 \\ a_3&=Kb_1 \end{aligned}$
& $K=\textstyle \frac{1}{b_3-b_1-1}>0$
& $|b_3|<|b_1|$
& $\begin{aligned} \dot x &= x^{K b_3}y^{b_1}-1 \\ \dot y &= K(1-x^{K b_1}y^{b_3}) \end{aligned}$ \\[2ex] \hline
\end{tabular}
\vspace{4ex}
\caption{Special cases of the ODE~\eqref{ode_K} with $K>0$ having a center.
The first four cases were already stated in Dancs\'o et al~\cite{dancso:farkas:farkas:szabo:1991},
where first integrals have been given.
The last two cases correspond to reversible systems.
}
\end{center}
\end{table}

\subsection{Case distinction} \label{ss:cd}

Using $\tr J=0$, that is, $a_1=K b_3$ by Equation~\eqref{eq:J},
we compute $\det J$ and the first two focal values, $L_1$ and $L_2$.
We find
\[
\det J =K (a_3 b_1 - b_3^2 K)
\]
and note that $\det J>0$ implies $a_3,b_1\neq0$.
Further,
using the Maple program in \cite{Kuznetsova:2012},
we find
\begin{equation*} \label{eq:l1}
L_1 = \frac{\pi}{8} \, \frac{K b_3 \left[ b_1 (1 + a_3 - a_3 K - b_3 K) - a_3 (1 - b_3) K \right]}{\sqrt{\det J} \, b_1} .
\end{equation*}
Expressions for $L_2$ (in case $L_1=0$) will be given below.

We show that all parameters $a_1,b_1,a_3,b_3 \in \R$ and $K>0$ in the ODE~\eqref{ode_K} for which
\begin{gather*}
\tr J=L_1=L_2=0 \\
\text{and } \det J>0
\end{gather*}
belong to one of the six cases in Table~\ref{table}.

To begin with, $L_1=0$ implies either
\begin{enumerate}
\item[(a)] $b_3=0$,
\item[(b)] $b_1=\frac{a_3 (1 - b_3) K}{D}$, where $D=1 + a_3 - a_3 K - b_3 K$ and $D\neq0$, or
\item[(c)] $D=0$ and $b_3 = 1$.
In this case, $(1+a_3)(1-K)=0$
and either
\begin{itemize}
\item[(c1)] $b_3=1$, $a_3=-1$  or
\item[(c2)] $b_3=1$, $K=1$.
\end{itemize}
\end{enumerate}

In case~(a), where $b_3=0$ (and hence $a_1=0$),
we find $\det J = K a_3 b_1$. 
Hence, the situation is covered by case~(i) in Table~\ref{table}.

In case~(b), where $D \neq 0$,
we find $b_3 \neq 1$ (otherwise $b_1=0$) 
and, using the Maple program in \cite{Kuznetsova:2012},
\[
L_2 = \frac{\pi}{288} \, \frac{(a_3 + b_3)^2 b_3 (1 + a_3 - b_3 K) (1-b_3 K) (1-K) (1 + a_3 + K - b_3 K)}{\sqrt{\det J}D(1-b_3)} .
\]
Now, $L_2=0$ implies that at least one of six factors is zero.
The first subcase $a_3+b_3=0$ implies $D=1-b_3$ and hence $b_1=a_3 K$ and $\det J = 0$.
As shown above, the subcase $b_3=0$ is covered by case~(i) in Table~\ref{table}.
The subcase $1 + a_3 - b_3 K=0$ implies $D=-a_3K$ and hence $b_1=b_3-1$.
That is, $a_1=a_3+1$, $b_3=b_1+1$, and hence $\det J=K(1-a_1-b_3)$ which corresponds to case~(ii).
The subcase $b_3 K=1$ (and hence $a_1=1$) implies $D=a_3(1-K)$ and hence $b_1=-1$.
Now, $\det J=-K(a_3+b_3)$, and the situation is covered by case~(iv).
The subcase $K=1$ (and hence $a_1=b_3$) implies $D=1-b_3$ and hence $b_1=a_3$.
Now, $\det J=b_1^2-a_1^2$, and the situation is covered by case~(r1).
Finally, the subcase $1 + a_3 + K - b_3 K=0$ implies $D=-(1+a_3)K$ and hence $b_1=\frac{a_3(1-b_3)}{-(1+a_3)}=\frac{a_3}{K}$.
That is, $a_1=Kb_3$, $a_3=Kb_1$ and hence $K=\frac{1}{b_3-b_1-1}$, $\det J=K^2(b_1^2-b_3^2)$ which corresponds to case~(r2).

In case (c1), where $b_3=1$ and $a_3=-1$ (and hence $a_1=K$), we find $\det J=-K(a_1+b_1)$. 
Hence, the situation is covered by case~(iii) in Table~\ref{table}.
In case (c2), where $b_3=1$ and $K=1$,
we find
\begin{equation} \label{eq:L2}
L_2 = \frac{\pi}{288} \, \frac{a_3 (1 + a_3) (1 + b_1) (a_3-b_1)}{\sqrt{\det J} \, b_1} .
\end{equation}
Now, $L_2=0$ implies that at least one of four factors is zero.
The first subcase $a_3=0$ implies $\det J<0$.
As shown above, the subcase $a_3=-1$ is covered by case~(iii).
Finally, the subcase $b_1=-1$ is covered by case~(iv),
and the subcase $a_3=b_1$ is covered by case~(r1).


\subsection{Reversible systems} \label{ss:rs}

Let $R: \R^2 \to \R^2$ be a reflection along a line.
A vector field $F \colon \R^2 \to \R^2$ (and the resulting dynamical system) is called reversible w.r.t.\ $R$ if
\[
F\circ R = - R \circ F .
\]
It is easy to see that, for any function $f \colon \R^2 \to \R$, the system 
\begin{align} \label{eq:rev}
\dot x &= f(x,y) \\
\dot y &= - f(y,x) \nonumber
\end{align}
is reversible w.r.t.\ the reflection $R \colon (x,y) \mapsto (y,x)$.
The following is a well-known fact, see e.g.\ \cite[4.6571]{nemytskii:stepanov:1960} or, more generally, \cite[Theorem~8.1]{devaney:1976}.

{\it An equilibrium of a reversible system which has purely imaginary eigenvalues and lies on the symmetry line of $R$ is a center.} 

Now we are in a position to deal with the last two cases in Table~\ref{table}.

Case (r1):
\begin{align*}
\dot x &= x^{a_1}y^{b_1}-1 ,\\
\dot y &= 1-x^{b_1}y^{a_1} .
\end{align*}                 
This vector field is of the form \eqref{eq:rev}, and hence it is reversible.

Case (r2):
\begin{align*}
\dot x &= x^{K b_3}y^{b_1}-1 , \\
\dot y &= K(1-x^{K b_1}y^{b_3}) ,
\end{align*}
where $K = \frac{1}{b_3-b_1-1}$.
We apply the coordinate transformation $u = x^K, v = y^{-1}$ and obtain
\begin{align*}
\dot u &= Kx^{K-1} ( x^{K b_3}y^{b_1}-1) = K u^{1-\frac{1}{K}} (u^{b_3}  v^{-b_1} - 1) , \\
\dot v &= - K y^{-2} (1-x^{K b_1}y^{b_3}) =- K v^2 (1-u^{b_1}v^{-b_3}) .
\end{align*}
Finally, we multiply the vector field with the positive function $K^{-1} v^{b_1}$ and obtain
\begin{align*}
\dot u &= u^{1-\frac1{K}+ b_3}   - u^{1-\frac1{K}}v^{b_1} , \\
\dot v &= -  v^{2 +b_1} +u^{b_1}v^{2+b_1-b_3} .
\end{align*}
Since $1-\frac{1}K = 2+b_1-b_3$, this vector field is of the form \eqref{eq:rev}, and hence it is reversible.                  

Since (r1) and (r2) lead to centers, analytic first integrals must exist. However, it seems difficult to find them.
So far we succeeded only in the intersection of (r1) and (r2), that is,
the case where $a_1 = b_3 = b_1 + 2$, $a_3=b_1$, $K=1$, and $b_1<-1$ (a one parameter family).
See Table~\ref{table:fi}.

\begin{table} \label{table:fi}
\begin{center}
\begin{tabular}{|c|c|c|c|}
\hline
case & ODE & first integral $V(x,y)$ & i.f.~$h(x,y)$ \\[2ex]
\hline\hline
(i) 
& $\begin{aligned} \dot x &= y^{b_1}-1 \\ \dot y &= K(1-x^{a_3}) \end{aligned}$
& $(x-\frac{1}{a_3+1}x^{a_3+1}) + \frac{1}{K}(y-\frac{1}{b_1+1}y^{b_1+1})$
& $1$ \\[2ex] \hline
(ii)
& $\begin{aligned} \dot x &= x^{a_1}y^{b_3-1}-1 \\ \dot y &= \tfrac{a_1}{b_3} (1-x^{a_1-1}y^{b_3}) \end{aligned}$ 
& $a_1 x+b_3 y - x^{a_1}y^{b_3}$
& $1$ 
\\[2ex] \hline
(iii) 
& $\begin{aligned} \dot x &= x^{a_1}y^{b_1}-1 \\ \dot y &= a_1 (1-\tfrac{y}{x}) \end{aligned}$ 
& $- \frac{a_1}{a_1-1}x^{-a_1+1}-\frac{1}{b_1+1}y^{b_1+1}+x^{-a_1}y$
& $x^{-a_1}$ \\[2ex] \hline
(iv) 
& $\begin{aligned} \dot x &= \tfrac{x}{y}-1 \\ \dot y &= \tfrac{1}{b_3}(1-x^{a_3}y^{b_3}) \end{aligned}$ 
& $- \frac{1}{a_3+1}x^{a_3+1}-\frac{b_3}{b_3-1}y^{-b_3+1}+xy^{-b_3} $
& $y^{-b_3}$ \\[2ex] \hline \hline
(r1)$\cap$(r2)
& $\begin{aligned} \dot x &= x^{b_1+2}y^{b_1}-1 \\ \dot y &= 1-x^{b_1}y^{b_1+2} \end{aligned}$
& $(\frac{1}{x}+\frac{1}{y})^{-(b_1+1)} (1 + (xy)^{-(b_1+1)})$
& $(x+y)^{-(b_1+2)}$ \\[2ex] \hline
\end{tabular}
\vspace{4ex}
\caption{Special cases of the ODE~\eqref{ode_K} with $K>0$ having a center 
and the corresponding first integrals and integrating factors (i.f.).
If $\al$ is zero in $\frac{x^\al}{\al}$ (in a first integral), replace $\frac{x^\al}{\al}$ by $\ln x$.
}
\end{center}
\end{table}


\subsection{Limit cycles}

As a simple consequence of our characterization of the center variety,
we can construct systems with two limit cycles via a degenerate Hopf or Bautin bifurcation,
see~\cite[Section~8.3]{kuznetsov:2004}.
We pick a system with $\tr J=L_1=0$ and $L_2 \neq 0$,
in particular, we consider case (c2) of our case distinction:
we take $b_3 =  a_1 = 1$, $K=1$ and hence $\tr J=L_1=0$
and choose $b_1$ and $a_3$ such that $L_2 <0$
with $L_2$ given by Equation~\eqref{eq:L2};
for example, $b_1 = -2$, $a_3 = -3$ or $b_1 = 2, a_3 = 1$.
If we now slightly perturb $K$ (keeping $a_1=K$) such that $L_1 > 0$ (and $\tr J=0$),
the resulting system has a stable limit cycle.
Finally, if we slightly change $a_1$ such that $\tr J< 0$, we create a small unstable limit cycle
via a subcritical Hopf bifurcation.

It remains open, if the ODE~\eqref{ode_K} admits more than two limit cycles.
For a computational algebra approach to this question, see~\cite{romanovski:shafer:2009}.




\section*{Acknowledgments}
BB and SM were supported by the Austrian Science Fund (FWF), project P28406.
GR was supported by the FWF, project P27229.

\section*{Supplementary material}
We provide a Maple worksheet
containing (i) the program from \cite{Kuznetsova:2012}
for the computation of the first two focal values
and (ii) the case distinction described in Section~\ref{ss:cd}.
Further we provide a CDF file (created with Mathematica)
containing 3-dimensional visualizations of the center variety.
(Thereby, we start from the 5 parameters $a_1$, $b_1$, $a_3$, $b_3$, and $K$,
use $\tr J=0$, that is, $a_1=K b_3$, and fix $K$.
As a result, we obtain plots in the 3 parameters $a_1$, $b_1$, and $a_3$.)

\bibliographystyle{abbrv}
\bibliography{stability}

\begin{thebibliography}{10}

\bibitem{dancso:farkas:farkas:szabo:1991}
A.~Dancs{\'o}, H.~Farkas, M.~Farkas, and G.~Szab{\'o}.
\newblock Investigations into a class of generalized two-dimensional
  {L}otka-{V}olterra schemes.
\newblock {\em Acta Appl. Math.}, 23(2):103--127, 1991.

\bibitem{devaney:1976}
R.~L. Devaney.
\newblock Reversible diffeomorphisms and flows.
\newblock {\em Trans. Amer. Math. Soc.}, 218:89--113, 1976.

\bibitem{farkas:noszticzius:1985}
H.~Farkas and Z.~Noszticzius.
\newblock Generalized {L}otka-{V}olterra schemes and the construction of
  two-dimensional explodator cores and their {L}iapunov functions via
  ``critical'' {H}opf bifurcations.
\newblock {\em J. Chem. Soc. Faraday Trans. II}, 81(10):1487--1505, 1985.

\bibitem{kuznetsov:2004}
Y.~A. Kuznetsov.
\newblock {\em Elements of applied bifurcation theory}, volume 112 of {\em
  Applied Mathematical Sciences}.
\newblock Springer-Verlag, New York, third edition, 2004.

\bibitem{Kuznetsova:2012}
O.~A. Kuznetsova.
\newblock An example of symbolic computation of {L}yapunov quantities in
  {M}aple.
\newblock In {\em Proceedings of the 5th WSEAS Congress on Applied Computing
  Conference, and Proceedings of the 1st International Conference on
  Biologically Inspired Computation}, BICA'12, pages 195--198, Stevens Point,
  Wisconsin, USA, 2012. World Scientific and Engineering Academy and Society
  (WSEAS).

\bibitem{lotka:1910}
A.~J. Lotka.
\newblock Contribution to the theory of periodic reactions.
\newblock {\em J. Phys. Chem.}, 14(3):271--274, 1910.

\bibitem{lotka:1920:a}
A.~J. Lotka.
\newblock Analytical note on certain rhythmic relations in organic systems.
\newblock {\em Proc. Natl. Acad. Sci.}, 6(7):410--415, 1920.

\bibitem{lotka:1920:b}
A.~J. Lotka.
\newblock Undamped oscillations derived from the law of mass action.
\newblock {\em J. Am. Chem. Soc.}, 42:1595--1599, 1920.

\bibitem{mueller:regensburger:2012}
S.~M\"uller and G.~Regensburger.
\newblock Generalized mass action systems: {C}omplex balancing equilibria and
  sign vectors of the stoichiometric and kinetic-order subspaces.
\newblock {\em SIAM J. Appl. Math.}, 72:1926--1947, 2012.

\bibitem{mueller:regensburger:2014}
S.~M\"uller and G.~Regensburger.
\newblock Generalized mass-action systems and positive solutions of polynomial
  equations with real and symbolic exponents.
\newblock In V.~P. Gerdt, W.~Koepf, E.~W. Mayr, and E.~H. Vorozhtsov, editors,
  {\em Computer Algebra in Scientific Computing. Proceedings of the 16th
  International Workshop (CASC 2014)}, volume 8660 of {\em Lecture Notes in
  Comput. Sci.}, pages 302--323, Berlin/Heidelbergx, 2014. Springer.

\bibitem{nemytskii:stepanov:1960}
V.~V. Nemytskii and V.~V. Stepanov.
\newblock {\em Qualitative Theory of Differential Equations}.
\newblock Princeton University Press, 1960.

\bibitem{romanovski:shafer:2009}
V.~G. Romanovski and D.~S. Shafer.
\newblock {\em The center and cyclicity problems: a computational algebra
  approach}.
\newblock Birkh\"auser Boston, Inc., Boston, MA, 2009.

\end{thebibliography}

\end{document}